\documentclass[12pt]{article}

\usepackage[utf8]{inputenc}
\usepackage[T1]{fontenc}
\usepackage{amsmath,amsthm,amsfonts,amssymb}
\usepackage{graphicx}
\usepackage{epstopdf}
\usepackage{booktabs}
\usepackage{float}
\usepackage{algorithm}
\usepackage{algorithmic}
\usepackage{hyperref}
\usepackage[margin=1in]{geometry}

\numberwithin{equation}{section}

\hypersetup{
  colorlinks=true,
  linkcolor=blue,
  citecolor=blue,
  urlcolor=blue,
  pdftitle={Minimization of a One Class of Maximum Functions with Applications to Some Eigenvalue Problems},
  pdfauthor={\.I. C. Ismayilov, T. B\"uy\"ukk\"oro\u{g}lu, V. Dzhafarov}
}

\newtheorem{theorem}{Theorem}[section]
\newtheorem{proposition}[theorem]{Proposition}

\theoremstyle{definition}
\newtheorem{example}[theorem]{Example}
\newtheorem{remark}[theorem]{Remark}

\title{Minimization of a One Class of Maximum Functions\\
with Applications to Some Eigenvalue Problems}

\author{
  \.Ipek Cafer Ismayilov\thanks{Department of Mathematics, Faculty of Science,
  Marmara University, 34722 Kad\i k\"oy \.Istanbul, T\"urkiye (ipek.ismayilov@marmara.edu.tr).
  The first author is supported by T\"{U}B\.{I}TAK (The Scientific and Technological Research Council of T\"urkiye) under the 2210-A National Graduate Scholarship Program.}
  \and
  Taner B\"uy\"ukk\"oro\u{g}lu\thanks{Department of Mathematics, Faculty of Science,
  Eski\c{s}ehir Technical University, 26555 Eski\c{s}ehir, T\"urkiye (\texttt{tbuyukkoroglu@eskisehir.edu.tr}, \texttt{vcaferov@eskisehir.edu.tr}).}
  \and
  Vakif Dzhafarov\footnotemark[2]
}

\date{}

\begin{document}

\maketitle

\begin{abstract}
In this paper we consider the problem of minimization of a convex function that
can be expressed as a maximum of affine linear functions. A convergence theorem is
established and the minimizing sequence is constructed. The computational time is
linear with respect to the step number. The obtained results are applied to eigenvalue
problems such as eigenvalue minimization, symmetric stabilization, and the Riccati
matrix inequalities. A number of examples are provided.
\end{abstract}



\section{Introduction}

Let $ X \subset \mathbb{R}^d $ be a convex and compact set, and let $ Y $ be a nonempty set. Suppose that $ a(y): Y \to \mathbb{R}^d $ and $ b(y): Y \to \mathbb{R} $ are bounded functions satisfying $ \|a(y)\| \leq K_1 $ and $ |b(y)| \leq K_2 $ for all $ y \in Y $. Consider the following convex, nondifferentiable function $\varphi: X \to \mathbb{R}$ defined by
\begin{equation} \label{denklem1}
  \varphi(x)=\max_{y \in Y} \left(\langle x,a(y) \rangle + b(y) \right).
\end{equation}
It is assumed that for all $x \in X$, the maximizing point $y$ in \eqref{denklem1} exists and can be computed easily.

We consider the problem of minimizing $\varphi(x)$ over $X$, that is, to find $x_{\mathrm{opt}} \in X$ such that
\begin{equation} \label{denklem2}
  \min_{x \in X} \varphi(x)=\varphi(x_{\mathrm{opt}}).
\end{equation}

A sequence $\{x^k\}$ is called a minimizing sequence if $\varphi(x^k)$ converges and
\[
  \lim_{k\to\infty} \varphi(x^k) = \min_{x\in X} \varphi(x).
\]

Given a convex function $\phi: \mathbb{R}^d \to \mathbb{R}$, a vector $g \in \mathbb{R}^d$ is called a subgradient of $\phi$ at $z=z_0$ if
\[
  \phi(z)-\phi(z_0) \geq \langle g, z-z_0 \rangle
\]
for all $z$ \cite{Rockafellar73}. The set of all subgradients of $\phi$ at $z_0$ is called the subdifferential of $\phi$ at $z_0$, denoted by $\partial \phi(z_0)$. It is well known that $\partial \phi(z_0) \neq \emptyset$.

Recall that a point $z_0 \in \mathbb{R}^d$ is called a local minimum of $\phi(z)$ if there exists an open set $U$ containing $z_0$ such that $\phi(z)\geq \phi(z_0)$ for all $z \in U$. An important property of convex functions is the following.
\begin{theorem}\label{thm11}
    Let $\phi: \mathbb{R}^d \to \mathbb{R}$ be convex. Then any local minimum point $z_0$ is the global minimum, that is
    \[
      \phi(z_0)=\min_{z \in \mathbb{R}^d} \phi(z).
    \]
\end{theorem}

Recall Kelley's cutting plane algorithm for convex minimization over a box. Consider the minimization problem of $\phi(z)$ over the box $Z=\{z \in \mathbb{R}^d: z_{\min} \leq z \leq z_{\max}\}$. Suppose $z^1, z^2, \dots, z^k \in Z$ and at least one subgradient $g^1 \in \partial \phi(z^1), \dots , g^k \in \partial \phi(z^k)$ is easy to calculate. Denote
\[
  \phi_k(z)=\max_{1 \leq i \leq k} \left[ \phi(z^i)+\langle g^i, z-z^i \rangle \right], \quad
  L_k=\min_{z \in B} \phi_k(z), \quad U_k=\min_{1 \leq i \leq k} \phi(z^i).
\]
Let $z^{k+1}$ be a minimizer in the definition of $L_k$, that is $L_k=\phi_k(z^{k+1})$.

$L_k$ is increasing, $U_k$ is decreasing and
\[
  L_k \leq \phi^* \leq U_k,
\]
where $\phi^*=\min_{z \in Z} \phi(z)$.
\begin{theorem}[\cite{Boyd91}]
  Let a tolerance number $\varepsilon>0$ be given. Then there exists $k$ such that
  \[
    0 \leq \phi(z^k)-\phi^* < \varepsilon.
  \]
\end{theorem}

Some improvement of Kelley's cutting-plane method has been done in \cite{Drori2015} where the minimization of a convex globally Lipschitz function over a ball is considered. The method in \cite{Drori2015} requires knowing in advance the number of iterations to be performed. In \cite{Bagirov2019} for constrained nonconvex optimization problems the augmented Lagrangian-based method is developed. At each step the discrete gradient method is applied to minimize the augmented Lagrangian.

In this paper, we consider the minimization problem over a box or over a polytope of a wider class of convex functions, which is the pointwise maximum of an affine-linear family. At each step of the iterative procedure the solution of a linear programming (LP) problem is required. It is shown that the sequence obtained from the solutions of these LP problems is a minimizing sequence. To show this, firstly we prove that every limit point of this sequence is an optimal point (Theorem \ref{theorem21}). Then, using Theorem \ref{theorem21} we prove the main theorem (Theorem \ref{theorem24}).

In the last section, the obtained results are applied to some eigenvalue problems, such as eigenvalue minimization, symmetric stabilization, Riccati matrix inequalities.

The results of this paper are an extension of the results on common Lyapunov functions obtained in \cite{Aksoy2025,Dzhafarov22,Dzhafarov2015}.

\section{Minimizing sequence}
\label{sec:convergence}

In this section, we give a minimization algorithm and a convergence theorem that solve the problem. This result improves Kelley's cutting-plane method.
\begin{proposition}
  Let $\hat{x} \in X$ be a given and $\hat{y} \in Y$ be the maximizing vector of $\langle \hat{x}, a(y) \rangle + b(y)$. Then for all $x \in X$ the following are satisfied
\begin{eqnarray}
\varphi(\hat{x}) = \langle \hat{x}, a(\hat{y})  \rangle + b(\hat{y}), \label{denklem3} \\ 
\varphi(x) \geq \langle x, a(\hat{y}) \rangle + b(\hat{y}). \label{denklem4}
\end{eqnarray}
\end{proposition}
\begin{proof}
 The proof follows directly from the definition of $\varphi(x)$ (see \eqref{denklem1}).
\end{proof}

As follows from \eqref{denklem3} and \eqref{denklem4}, the vector $a(\hat{y})$ is a subgradient vector of $\varphi(x)$ at $x=\hat{x}$ since for all $x \in X$
\[
  \varphi(x) - \varphi(\hat{x}) \geq \langle x-\hat{x} , a(\hat{y}) \rangle.
\]

Now construct an increasing scalar sequence $\{s_k\}$ as follows.

Starting from any $x^1 \in X$, define the vector $y^1 \in Y$ that maximizes
\[
  \langle x^1, a(y) \rangle + b(y),
\]
that is, $\varphi(x^1)=\langle x^1, a(y^1) \rangle + b(y^1)$, and define
\begin{equation} \label{denklem5}
  s_1=\min_{x \in X} \left( \langle x, a(y^1) \rangle + b(y^1)\right).
\end{equation}
Let the minimizing vector in \eqref{denklem5} be $x^2 \in X$, and let the maximizing vector of
\[
  \langle x^2, a(y) \rangle + b(y)
\]
be $y^2 \in Y$. Define
\begin{equation} \label{denklem6}
  s_2=\min_{x \in X} \max_{1 \leq i \leq 2} \left[ \langle x, a(y^i) \rangle + b(y^i) \right].
\end{equation}
Let the minimizing vector in \eqref{denklem6} be $x^3 \in X$, and let the maximizing vector of
\[
  \langle x^3, a(y) \rangle + b(y)
\]
be $y^3 \in Y$, and define
\[
  s_3=\min_{x \in X} \max_{1 \leq i \leq 3} \left[ \langle x,a(y^i) \rangle + b(y^i) \right].
\]
Continuing, at the $k$-th step define
\begin{equation} \label{denklem7}
  s_k=\min_{x \in X} \max_{1 \leq i \leq k} \left[ \langle x,a(y^i) \rangle + b(y^i) \right]
\end{equation}

In the sequel, we assume that the set $X$ is defined by linear inequalities. Then the number $s_k$ and the minimizing vector $x^{k+1}$ in \eqref{denklem7} can be evaluated by solving an LP problem (see \cite{Boyd91,Dzhafarov22}). In the following proposition, we describe the computation of $s_k$ and $x^{k+1}$ by LP problem in the case when $X$ is a box:
\[
  X=B_a=\{(x_1,x_2,\dots,x_d): \ |x_i| \leq a, \, i=1,2,\dots,d\}.
\]
\begin{proposition} [\cite{Boyd91,Dzhafarov22}]\label{thm:lpform}
Let $k$ be fixed. Consider the following LP problem:
\[
  \left\{
  \begin{aligned}
    &c(x,t)=t \to \min, \\
    &-a \leq x_j \leq a,\quad j=1,2,\dots,d, \\
    &-\sigma \leq t \leq \sigma,\quad \sigma=a\sqrt{d}K_1+K_2, \\
    &\langle x, a(y^i)\rangle + b(y^i) \leq t,\quad i=1,2,\dots,k.
  \end{aligned} \right.
\]
If $(x^*,t^*)$ is an optimal solution of this LP problem, then $s_k=c(x^*,t^*)=t^*$, and $x^*$ is a minimizer in \eqref{denklem7}.
\end{proposition}
This proposition shows that each step of the procedure reduces to solving an LP problem, which forms the basis of the algorithm presented next.

The following observations can be made easily.
\begin{itemize}
  \item[\textit{a})] The sequence $\{s_k\}$ is monotone increasing and bounded, and therefore has a limit $s_k \to s$.
  \item[\textit{b})] $\varphi(x) \geq s_k$ for all $x \in X$ and $k=1,2,3,\dots$.
\end{itemize}

Since $X$ is compact, the sequence $\{x^k\}$ has a convergent subsequence. In the sequel, we prove that any limit point of the sequence $\{x^k\}$ (i.e. the limit of a convergent subsequence of $\{x^k\}$) is a minimizing point of $\varphi(x)$, and that $\{x^k\}$ is a minimizing sequence for the problem \eqref{denklem1}-\eqref{denklem2}.
\begin{theorem}\label{theorem21}
  Let $\tilde{x} \in X$ be any limit point of $\{x^k\}$ (recall that $x^{k+1}$ is the minimizer in \eqref{denklem7}). Then $\tilde{x}$ is a minimizing point of $\varphi(x)$, that is,
  \[
    \min_{x \in X} \varphi(x) = \varphi(\tilde{x}).
  \]
\end{theorem}
\begin{proof}
  Consider natural numbers $p$ and $q$ with $q \geq p$. Assume that $y^p \in Y$ maximizes $\langle x^p, a(y) \rangle + b(y)$ and $y^q$ maximizes $\langle x^q, a(y) \rangle + b(y)$. By \eqref{denklem7}
  \begin{equation} \label{denklem8}
    s_q=\min_{x \in X} \max_{1 \leq i \leq q} \left( \langle x, a(y^i) \rangle + b(y^i)\right).
  \end{equation}
  Let the minimizer in \eqref{denklem8} be $x^{q+1} \in X$:
  \[
    s_q = \max_{1 \leq i \leq q} \left[ \langle x^{q+1}, a(y^i) \rangle + b(y^i)\right].
  \]
  The sequence $\{s_q\}$ is increasing and bounded, therefore it has limit $s=\lim_{q \to \infty} s_q$. The following can be written:
  \begin{align*}
      & s_q \geq \langle x^{q+1}, a(y^p) \rangle + b(y^p), \quad s_q \leq s,\\
      & s \geq \langle x^{q+1}, a(y^p) \rangle + b(y^p),\\
      & \varphi(x^p)= \langle x^p, a(y^p) \rangle + b(y^p).
\end{align*}
Combining these relations yields
  \begin{equation} \label{denklem9}
    0 \geq s-\varphi(x^p) \geq \langle x^{q+1} - x^p , a(y^p) \rangle.
  \end{equation}
  Let $\{x^{k_m}\}$ be a subsequence of $\{x^k\}$ which converges to $\tilde{x}$: $x^{k_m} \to \tilde{x}$ as $m \to \infty$. We show that $\tilde{x}$ is an optimal point. Take $p=k_m$, $q+1=k_{m+1}$ in \eqref{denklem9}, then
  \begin{equation} \label{denklem10}
      0 \geq s-\varphi(x^{k_m}) \geq \langle x^{k_{m+1}}-x^{k_m}, a(y^{k_m}) \rangle.
  \end{equation}
  Take the limit as $m \to \infty$ in \eqref{denklem10}. Since $\varphi(x)$ is continuous, $\{x^{k_m}\}$ is convergent and $\|a(y^{k_m})\| \leq K_1$, we have
  \begin{gather*}
  |\langle x^{k_{m+1}}-x^{k_m}, a(y^{k_m}) \rangle| 
    \leq \|x^{k_{m+1}}-x^{k_m}\|\cdot \|a(y^{k_m})\| 
    \leq K_1 \|x^{k_{m+1}}-x^{k_m}\| \to 0,\\
  0 \geq s - \varphi(\tilde{x}) \geq 0 \;\;\Rightarrow\;\; \varphi(\tilde{x})=s.
\end{gather*}
On the other hand, from $\varphi(x) \geq s_k$ it follows that $\varphi(x) \geq s$ for all $x \in X$, therefore
  \[
    \min_{x \in X} \varphi(x) = \varphi(\tilde{x}).
  \]
\end{proof}
Now, using Theorem \ref{theorem21} we prove that $\{x^k\}$ is a minimizing sequence.
\begin{theorem}\label{theorem24}
  The sequence obtained above $\{x^k\}$ is a minimizing sequence for $\varphi(x)$, that is, 
  $\varphi(x^k)$ has a limit as $k \to \infty$ and
  \[
    \lim_{k \to \infty} \varphi(x^k) = \min_{x \in X} \varphi(x).
  \]
\end{theorem}
\begin{proof}
    Since $X$ is compact, $a(y)$ and $b(y)$ are bounded, the sequence $\varphi(x^k)$ is bounded.
    We show that any convergent subsequence of $\varphi(x^k)$ converges to the same limit $\min_{x \in X} \varphi(x)$. Let $\varphi(x^{k_l})$ be any convergent subsequence of $\varphi(x^k)$. The sequence $\{x^{k_l}\}$ is bounded and has a convergent subsequence $x^{k_{l_j}}\to a$ as $j \to \infty$. By Theorem~\ref{theorem21}, $a$ is an optimal point:
    \[
      \varphi(a)=\min_{x \in X} \varphi(x).
    \]
    On the other hand, by continuity $\varphi(x^{k_{l_j}}) \to \varphi(a)$ as $j \to \infty$. As a result, $\varphi(a)$ is a limit point of the convergent sequence $\varphi(x^{k_l})$. Therefore $\varphi(x^{k})$ converges to $\varphi(a)$:
    \[
      \lim_{k\to \infty} \varphi(x^k)=\varphi(a)=\min_{x\in X} \varphi(x).
    \]
\end{proof}

For the minimization of $\varphi : X \to \mathbb{R}$ defined by \eqref{denklem1}, the following algorithm can be proposed:

\begin{algorithm}[H]
\caption{}
\label{algorithm21}
\begin{algorithmic}[1]
  \STATE Calculate $K_1$ and $K_2$.
  \STATE Choose a tolerance number $\varepsilon > 0$ and any initial point $x^1 \in X$.
  \FOR{$k = 1,2,3,\dots$}
    \STATE Construct the sequences $s_k$ and $x^k$ by solving the corresponding LP problems (see Proposition~\ref{thm:lpform}).
    \IF{$\varphi(x^k) - s_k < \varepsilon$ for some $k$}
      \STATE \textbf{Stop}; $x^k$ is the minimizer.
    \ENDIF
  \ENDFOR
\end{algorithmic}
\end{algorithm}

\begin{remark}
  As follows from Proposition \ref{thm:lpform}, the number of variables $d$ is constant, whereas the number of constraints is $k+2d+1$, where $k$ is the step number. Then, according to \cite{Megiddo1984}, the computational time of the Algorithm \ref{algorithm21} is linear in the step number $k$. Interior-point polynomial time algorithms for general convex programming were given in \cite{Boyd04,Nesterov94}. 
\end{remark}

\begin{example}\label{example27}
Consider the following problem
\[
\begin{array}{c}
\varphi(x) \to \min, \quad x=(x_1,x_2) \in \mathbb{R}^2,\\
\displaystyle
\varphi(x)=\max_{(y_1,y_2)\in[-5,5]^2}
\left(x_1\sin(y_1-y_2) + x_2\cos(y_1+y_2) + y_1^2 - y_2^2\right).
\end{array}
\]
With $a=10$, $d=2$, $K_1=1.414$, and $K_2=24.997$, we calculate $\sigma=44.991296$. Take the tolerance number $\varepsilon=10^{-6}$. Starting from $x^1=(-2,1)$, the iteration results are summarized in 
Table~\ref{tab:example27}. 
\begin{table}[H]
\centering
\caption{Iteration results for Example~\ref{example27}}
\label{tab:example27}
\begin{tabular}{ccccc}
\toprule
$k$ & $x^{k}$ & $y^{k}$ & $\varphi(x^{k})$ & $s_k$ \\ 
\midrule
1  & $(-2,1)$ & $(5,0.345)$ & $27.4689352$ & $8.9817384$ \\
2  & $(10,-10)$ & $(-4.382,0.707)$ & $36.6120351$ & $20.2355678$ \\
3  & $(10,9.024)$ & $(-5,-0.676)$ & $41.2090678$ & $22.2390053$ \\
$\vdots$ & $\vdots$ & $\vdots$ & $\vdots$ & $\vdots$ \\
23 & $(0,-0.579)$ & $(-5,0.290)$ & $24.9172969$ & $24.9172758$ \\
24 & $(0,-0.574)$ & $(5,0.287)$ & $24.9172766$ & $24.9172762$ \\
\bottomrule
\end{tabular}
\end{table}

Thus Algorithm~\ref{algorithm21} yields the minimum value
\[
  \min_{x\in[-10,10]^2}\varphi(x)\approx \varphi(x^{24})=24.9172766.
\]
On the other hand $x^{24}=(0,-0.574)$ is an interior point of $[-10,10]^2$, therefore by Theorem \ref{thm11}
\[
 \min_{x\in\mathbb{R}^2}\varphi(x)=\min_{x\in[-10,10]^2} \varphi(x)\approx 24.9172766.
\]
\end{example}

\section{Some applications}
\label{sec:applications}

In this section we give applications of the obtained results to some eigenvalue problems.

\subsection{Eigenvalue minimization}

Let $A_0,A_1,\dots,A_d$ be $n\times n$ dimensional symmetric matrices, and $x=(x_1,\dots,x_d)$ be the decision vector. Consider the matrix function
\[
  x \mapsto A(x) := A_0 + x_1 A_1 + \cdots + x_d A_d,
\]
and the following convex, non-smooth minimization problem
\begin{equation}\label{eq31}
  \varphi(x) = \lambda_{\max}(A(x)) \to \inf, \, x \in \mathbb{R}^d.
\end{equation}

Optimization problems of the form above arise, for example, in mechanics, stability analysis, and combinatorics (see \cite{Jarre93}).  

Define 
\[
  \lambda_{\mathrm{opt}} = \inf \{ \varphi(x): x \in \mathbb{R}^d \}.
\]
The following statements are well known:
\begin{enumerate}
  \item[\textit{i})] If $\lambda_{\mathrm{opt}}$ is finite, then $\lambda_{\mathrm{opt}} \in [\lambda_{\min}(A_0),\lambda_{\max}(A_0)]$.
  \item[\textit{ii})] $\lambda_{\mathrm{opt}}$ is finite if and only if there is no vector $x$ such that $\sum_{i=1}^d x_i A_i$ is positive definite.
\end{enumerate}

Take $a>0$ and define
\begin{equation} \label{eqs32}
  \lambda_a=\min \{\varphi(x): x \in B_a\},
\end{equation}
where
\[
  B_a=\{x=(x_1,\dots,x_d): |x_i| \leq a, i=1,2,\dots,d\}.
\]
Obviously $\lambda_{\mathrm{opt}} \leq \lambda_a$.

The function $\lambda_{\max}(A(x))$ has the form $(\ref{eq31})$ with $Y = \{ y \in \mathbb{R}^n : \|y\|=1 \}$. Indeed,
\begin{align*}
    \varphi(x) &= \lambda_{\max}(A(x)) 
= \max_{\|y\|=1} y^T \left(A_0 + x_1 A_1 + \cdots + x_d A_d\right) y\\
&= \max_{\|y\|=1} \left( \langle x, a(y)\rangle + b(y) \right),
\end{align*}
where
\[
a(y) = \left(y^T A_1 y, \dots, y^T A_d y\right), 
\qquad b(y) = y^T A_0 y.
\]
Since
\[
\begin{array}{ll}
  \|a(y)\| = \sqrt{(y^T A_1 y)^2 + \cdots + (y^T A_d y)^2}
  & \leq |y^T A_1 y| + \cdots + |y^T A_d y|\\
  & \leq \|A_1\|_F + \cdots + \|A_d\|_F,
\end{array}
\]
and
\[
  |b(y)| = |y^T A_0 y| \leq \|A_0\|_F,
\]
where $\|A\|_F$ denotes the Frobenius matrix norm \cite{HornJohnson94}. Therefore (see the Introduction)
\[
  K_1 = \|A_1\|_F + \cdots + \|A_d\|_F, 
  \qquad K_2 = \|A_0\|_F.
\]
\begin{proposition}\label{prop3.1}
\vspace{0.2em} 
\begin{enumerate}
  \item[1)] $\lambda_a \to \lambda_{\mathrm{opt}}$ as $a \to \infty$.
  \item[2)] If $\lambda_{a_*} < \lambda_{\min}(A_0)$ for some $a_*>0$, then $\lambda_{\mathrm{opt}}=-\infty$.
  \item[3)] If for some $a_*>0$ the minimum in \eqref{eqs32} is attained at some interior point of $B_{a_*}$, then 
  $\lambda_{\mathrm{opt}}=\lambda_{a_*}$.
\end{enumerate}
\end{proposition}
\begin{proof}
1) is obvious. 2) follows from \textit{i}). 3) follows from Theorem~\ref{thm11}.
\end{proof}
Summarizing, the following simple algorithm can be proposed.
\begin{algorithm}[H]
\caption{}
\label{algorithm31}
\begin{algorithmic}[1]
  \STATE Take $a>0$ and solve \eqref{eqs32} using Algorithm \ref{algorithm21}.
  \IF{the minimum is attained at some interior point of $B_a$}
    \STATE \textbf{Stop}; the obtained $x^k$ is the minimizing sequence.
  \ELSE
    \STATE If the minimum is attained at a boundary point of $B_a$, then increase $a$ until the minimum in \eqref{eqs32} is attained at some interior point of $B_a$ or until $\lambda_a < \lambda_{\min}(A_0)$. In the last case, $\lambda_{\mathrm{opt}}=-\infty$.
  \ENDIF
\end{algorithmic}
\end{algorithm}

\begin{example}\label{example32}
Consider the matrices
\[
    A_0=\begin{bmatrix}
        1 & -1 \\
        -1 & 1
    \end{bmatrix},\quad
    A_1=\begin{bmatrix}
        1 & 2 \\
        2 & -1
    \end{bmatrix},\quad
    A_2=\begin{bmatrix}
        -3 & 4 \\
        4 & -1
    \end{bmatrix},\quad
    A_3=\begin{bmatrix}
        2 & -1 \\
        -1 & -1
    \end{bmatrix}.
\]
Define
\[
A(x)=A_0+x_1A_1+x_2A_2+x_3A_3, \qquad \varphi(x)=\lambda_{\max}(A(x)).
\]
We investigate the minimization of $\varphi(x)$ over $\mathbb{R}^3$. Take $X=B_1=\{(x_1,x_2,x_3): x \in [-1,1]^3\}$, ($a=1$) and the result of application of Algorithm~\ref{algorithm31} with the initial point $x^1=(1,0,0)\in X$ and tolerance number $\varepsilon=10^{-6}$ are presented in Table~\ref{tablo2}.
\begin{table}[H]
\centering
\caption{Iteration results for Example~\ref{example32}}
\label{tablo2}
\begin{tabular}{cccc}
\toprule
$k$ & $x^{k}$ & $\varphi(x^{k})$ & $s_k$ \\
\midrule
1 & $(1,0,0)$        & $2.41421356$  & $-2.80330086$ \\
2 & $(-1,-1,-1)$     & $8.68465844$  & $-2.56130792$ \\
3 & $(-1,0.994,-1)$  & $2.52660709$  & $-0.16666667$ \\
4 & $(-0.666,0.833,1)$ & $-0.16666667$ & $-0.16666667$ \\
\bottomrule
\end{tabular}
\end{table}

Since $\varphi(x^4)=-0.16666667<\lambda_{\min}(A_0)=0$, it follows from Proposition~\ref{prop3.1} that $\lambda_{\mathrm{opt}}=-\infty$.
\end{example}

\begin{example}\label{example33}
Consider $A(x)=A_0+x_1A_1+x_2A_2$, where
\begin{align*}
    A_0&=\begin{bmatrix}
         2 & -3 &  2 &  1 \\
        -3 & -4 & -1 &  0 \\
         2 & -1 & -1 & -2 \\
         1 &  0 & -2 &  3
    \end{bmatrix},\quad
    A_1=\begin{bmatrix}
         1 &  0 & -1 &  2 \\
         0 &  2 &  4 &  0 \\
        -1 &  4 & -3 & -1 \\
         2 &  0 & -1 &  1
    \end{bmatrix}, \\
    A_2&=\begin{bmatrix}
        -3 &  1 &  1 & -1 \\
         1 & -1 &  0 & -1 \\
         1 &  0 & -1 &  0 \\
        -1 & -1 &  0 &  2
    \end{bmatrix}.
\end{align*}
We minimize $\varphi(x)=\lambda_{\max}(A(x))$ over $X=B_1=\{(x_1,x_2): \ x\in [-1,1]^2\}$ by applying Algorithm~\ref{algorithm31}. With the initial point $x^1=(1,0)\in X$ and the tolerance number $\varepsilon=10^{-6}$, the iteration results are presented in Table~\ref{tab:lambda_max_4x4}.
\begin{table}[H]
\centering
\caption{Iteration results for Example~\ref{example33}}
\label{tab:lambda_max_4x4}
\begin{tabular}{cccc}
\toprule
$k$ & $x^{k}$ & $\varphi(x^{k})$ & $s_k$ \\ 
\midrule
1  & $(1,0)$   & $7.4280899$ & $-1.1242361$ \\
2  & $(-1,1)$  & $6.7242701$ & $3.1004383$ \\
3  & $(0.165,1)$   & $5.8176082$ & $3.7335785$ \\
$\vdots$ & $\vdots$ & $\vdots$ & $\vdots$ \\
10 & $(0.012,0.085)$   & $3.9849441$ & $3.9849440$ \\
\bottomrule
\end{tabular}
\end{table}
\begin{figure}[H]
  \centering
  \includegraphics[width=0.5\textwidth]{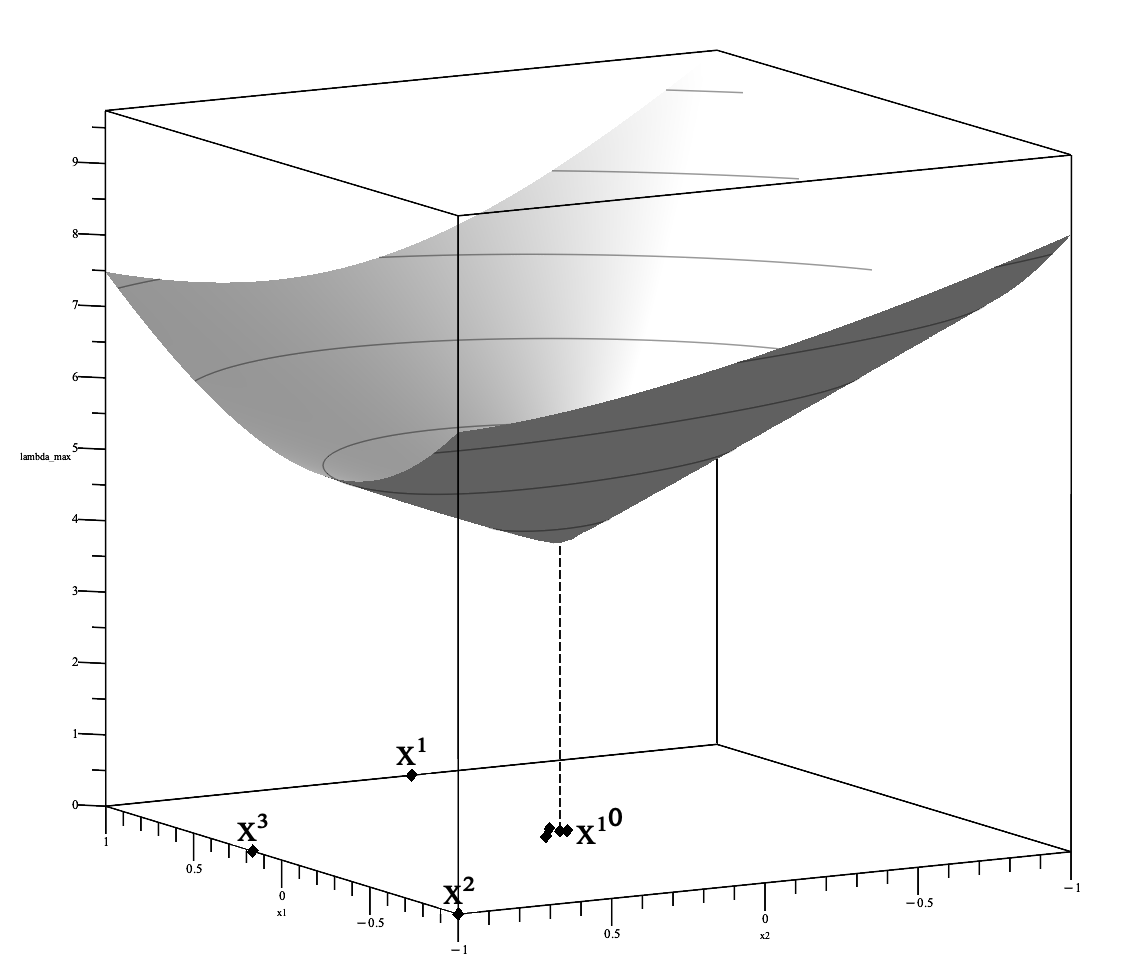}
  \caption{The graph of $\varphi(x)=\lambda_{\max}(A(x))$ for Example~\ref{example33}.}
  \label{figure1}
\end{figure}
Thus, $\lambda_{\mathrm{opt}}=3.9849441$ is attained at $x^{10}=(0.012,0.085)$. Moreover,
\[
    \lambda_{\mathrm{opt}} \in [\lambda_{\min}(A_0),\lambda_{\max}(A_0)]=[-5.262,4.208].
\]

By Theorem~\ref{thm11}
\[
  \min_{x \in \mathbb{R}^2} \varphi(x) = \min_{x \in [-1,1]^2} \varphi(x)=3.9849441.
\]
\end{example}

\section{Symmetric stabilization with maximal stability margin}
\label{sec:stabilization}

Given the control system
\[
  \dot{x} = A x + B u
\]
with unknown linear feedback $u=Kx$, where $A \in \mathbb{R}^{n \times n}$ and $B \in \mathbb{R}^{n \times m}$, consider the following problem:  

Find $K \in \mathbb{R}^{m \times n}$ such that $A+BK$ is symmetric and stable (i.e., $\lambda_{\max}(A+BK)<0$), and
\[
  \lambda_{\max}(A+BK) \;\to\; \min 
  \quad \text{subject to } \|K\|_\infty \leq a,
\]
where $a>0$ is sufficiently large and 
\[
  \|K\|_\infty = \max_{i,j} |k_{ij}|.
\]

For a stable (unstable) matrix, maximal perturbation bounds that preserves stability (unstability) properties have been investigated in numerous works (see, for example, \cite{Kalinina2022} and references therein).

After the parametrization
\[
K=
\begin{bmatrix}
x_{1} & x_{2} & \cdots & x_{n} \\
x_{n+1} & x_{n+2} & \cdots & x_{2n} \\
\vdots & \vdots & \ddots & \vdots \\
x_{d-n+1} & x_{d-n+2} & \cdots & x_{d}
\end{bmatrix}, \qquad (d=mn),
\]
the function $x \mapsto \lambda_{\max}(A+BK)$ is transformed into the form \eqref{denklem1} with $Y = \{\, y \in \mathbb{R}^n : \|y\|=1 \,\}$. The symmetricity condition of $A+BK$ yields $\tfrac{n(n-1)}{2}$ linear equalities in $x$. Denote this set by $\tilde{Z}$, that is
\[
  \tilde{Z} = \{ x=(x_1,x_2,\dots,x_d) \in \mathbb{R}^d : A+BK \ \text{is symmetric} \}.
\]
The function $\varphi(x)$ to be minimized for this problem is
\begin{equation}\label{eq:3.3}
  \varphi(x) = \lambda_{\max}(A+BK)
  = \max_{\|y\|=1} \left( \langle x,a(y)\rangle + y^T A_0 y \right),
\end{equation}
where $a(y) = \mathrm{vec}(B^Tyy^T)$ and $\mathrm{vec}(\cdot)$ denotes the vectorization operator that stacks the columns of a matrix into a single vector.

The function \eqref{eq:3.3} should be minimized on the convex set $X = B_a \cap \tilde{Z}$. For the global minimum, Theorem~\ref{thm11} is not applicable in this case, since the set $X$ has empty interior in $\mathbb{R}^d$.

Now compute the bounds $K_1$ and $K_2$. Obviously $K_2 = \|A_0\|_F$. On the other hand
\[
\begin{array}{ll}
\|uv^T\|_F^2 &= \mathrm{tr}\left[(uv^T)^T(uv^T)\right]
= \mathrm{tr}\left[(u^Tu)(vv^T)\right]
=(u^Tu)\,\mathrm{tr}(vv^T)\\
&= (u^Tu)(v^Tv) = \|u\|^2 \|v\|^2,
\end{array}
\]
for any $u,v \in \mathbb{R}^n$. Applying this property to $a(y)$ with $u=B^Ty$, $v=y$ we obtain
\[
\|a(y)\| = \|\mathrm{vec}(B^Tyy^T)\| 
= \|B^Tyy^T\|_F = \|B^Ty\|\|y\|.
\]
Since $\|y\|=1$, it follows that
\[
\|a(y)\| \leq \max_{\|y\|=1} \|B^Ty\|.
\]  
On the other hand $\max_{\|y\|=1} \|B^Ty\|$ is the spectral norm of the matrix $B$ which is less than or equal to the Frobenius norm of $B$. Thus one can choose
\[
K_1 = \|B\|_F = \sqrt{\mathrm{tr}(B^TB)}.
\]

\begin{example}\label{example41}
Consider the control system $\dot{x}=Ax+Bu$, where
\[
A=\begin{bmatrix}
3 & 1\\
4 & -3
\end{bmatrix},\quad
B=\begin{bmatrix}
-1 & 3 & 1\\
3 & -1 & 2
\end{bmatrix}.
\]
The feedback matrix $K$ has dimension $3\times2$, therefore $d=6$. The minimization is performed over $X=B_{1}=[-1,1]^6$. With the initial point $x^1=(1,0,0,0,0,0)$ and the tolerance number $\varepsilon=10^{-6}$, Algorithm~\ref{algorithm21} converges at the $13$-th iteration, where $\varphi(x^{13})-s_{13}\approx0$. The iteration results are summarized in Table~\ref{table41}.
\begin{table}[H]
\centering
\caption{Iteration results for Example~\ref{example41} with $a=1$}
\label{table41}
\begin{tabular}{cccc}
\hline
$k$ & $x^k$ & $\varphi(x^k)$ & $s_k$ \\
\hline
1 & $(1,0,0,0,0,0)$ & $3.14005494$ & $-4.33293514$\\
2 & $(-1,-1,-1,-1,-0.3333,-1)$ & $0.00000000$ & $-1.38629477$\\
3 & $(0.3863,-1,-1,-1,1,-0.8411)$ & $0.49782436$ & $-0.57887774$\\
4 & $(-0.4211,-1,-1,-1,0.2455,-1)$ & $-0.20384567$ & $-0.27594062$\\
\vdots & \vdots & \vdots & \vdots \\
12 & $(-0.5471,-1,-1,-1,0.1196,-1)$ & $-0.21923072$ & $-0.21923341$\\
13 & $(-0.5483,-1,-1,-1,0.1184,-1)$ & $-0.21923200$ & $-0.21923206$\\
\hline
\end{tabular}
\end{table}
At iteration $k=13$, the feedback matrix is
\[
K=\begin{bmatrix}
-0.5483 & -1\\
-1 & 0.1184\\
-1 & -1
\end{bmatrix},
\]
which yields
\[
A+BK=
\begin{bmatrix}
-0.4517 & 1.3551\\
1.3551 & -8.1184
\end{bmatrix},
\]
whose eigenvalues are $-0.2192$ and $-8.3509$. Hence, the closed-loop matrix $A+BK$ is stable.
\end{example}

\begin{example}\label{example42}
Consider the control system $\dot{x}=Ax+Bu$, where
\[
A=\begin{bmatrix}
 2 & 2 & 4\\
 1 & -2 & 3\\
 5 & 1 & 1
\end{bmatrix},\quad
B=\begin{bmatrix}
 1 & 2\\
 4 & 0\\
-1 & 1
\end{bmatrix}.
\]
The feedback matrix $K$ has dimension $2\times3$, therefore $d=3$. Take the minimization over $X=B_{10}=[-10,10]^3$. With the initial point $x^1=(1,0,0)$ and the tolerance number $\varepsilon=10^{-6}$, Algorithm~\ref{algorithm21} terminates at the $15$-th iteration since $\varphi(x^{15})-s_{15} \approx 0$. The iteration results are presented in Table~\ref{table42}.
\begin{table}[H]
\centering
\caption{Iteration results for Example~\ref{example42} with $a=10$}
\label{table42}
\begin{tabular}{cccc}
\hline
$k$ & $x^k$ & $\varphi(x^k)$ & $s_k$ \\
\hline
1 & $(1,0,0,0,0,0)$ & $7.46757500$ & $-41.15952669$\\
2 & $(-7.25, -10, -10, -10, -0.499, -0.625)$ & $1.20689918$ & $-14.83979132$\\
3 & $(0.523, -10, -10, 5.547, 3.386, -6.455)$ & $-1.03179807$ & $-1.42111053$\\
4 & $(-5.976, -10, -10, -7.453, 0.136, -1.579)$ & $-0.32905942$ & $-1.32359330$\\
\vdots & \vdots & \vdots & \vdots \\
14 & $(-1.138, -10, -10, 2.222, 2.555, -5.208)$ & $-1.07782050$ & $-1.07782181$\\
15 & $(-1.132, -10, -10, 2.234, 2.558, -5.212)$ & $-1.07782123$ & $-1.07782132$\\
\hline
\end{tabular}
\end{table}
The corresponding feedback matrix is
\[
K=\begin{bmatrix}
-1.132 & -10     &  2.558\\
-10    &   2.234 & -5.212
\end{bmatrix},
\]
which yields
\[
A+BK=
\begin{bmatrix}
 -19.132 &  -3.530 & -3.867\\
  -3.530 & -42     & 13.234\\
  -3.867 &  13.234 & -6.771
\end{bmatrix},
\]
with eigenvalues $-46.585$, $-20.241$, and $-1.077$. 
Thus, the matrix $A+BK$ is stable.
\end{example}

\section{Riccati Matrix Inequality}

Consider the Riccati matrix inequality
\begin{equation}\label{eq:riccati}
A^TP + PA - PBQ^{-1}B^TP + R > 0,
\end{equation}
where $A,P,R \in \mathbb{R}^{n\times n}$, $B \in \mathbb{R}^{n\times m}$, 
$Q \in \mathbb{R}^{m\times m}$, $Q>0$. 
In this inequality the matrix $P>0$ is unknown. We assume that $\lambda_{\min}(R)<0$, otherwise the solution can be evaluated easily. By the Schur Complement Theorem \cite{Boyd04,Polyak21}, the nonlinear inequality \eqref{eq:riccati} is equivalent to the following linear matrix inequality of dimension $(m+n)\times(m+n)$:
\[
f_1(P) = 
-\begin{bmatrix}
A^TP+PA & PB \\
B^TP    & 0
\end{bmatrix}
-
\begin{bmatrix}
R & 0 \\
0 & Q
\end{bmatrix} < 0.
\]
Adding $f_2(P)=-P<0$, we have to solve two matrix inequalities
\begin{equation}\label{eq:35}
\lambda_{\max}(f_1(P))<0,\qquad \lambda_{\max}(f_2(P))<0,
\end{equation}
with respect to the symmetric matrix $P$. Let
\[
P=P(x)=
\begin{bmatrix}
x_1 & x_2 & \cdots & x_n \\
x_2 & x_{n+1} & \cdots & x_{2n-1} \\
\vdots & \vdots & \ddots & \vdots \\
x_n & x_{2n-1} & \cdots & x_d
\end{bmatrix}, 
\qquad \left(d=\tfrac{n(n+1)}{2}\right),
\]
be a parametrization of $P$. Then \eqref{eq:35} can be written as
\begin{equation}\label{eq:36}
\varphi(x)=\max_{i=1,2}\, \lambda_{\max}\!\left(f_i(P(x))\right) < 0.
\end{equation}
We have
\begin{align} \label{eq:37}
    \lambda_{\max}\left(f_1(P(x))\right) &= \max_{\|u\|=1} \left(u^T
      \begin{bmatrix}
        -A^TP(x)-P(x)A & -P(x)B \\ \notag
        -B^TP(x) & 0
      \end{bmatrix} u
      + u^T
      \begin{bmatrix}
        -R & 0 \\
        0  & -Q
      \end{bmatrix} u \right)\\
      &=\max_{\|u\|=1}\left(a_1x_1+\cdots+a_d x_d+u^T
      \begin{bmatrix}
        -R & 0 \\
        0  & -Q
      \end{bmatrix} u\right),
\end{align}
and
\begin{equation} \label{eq:3.8}
\lambda_{\max}\left(f_2(P(x))\right)
= \max_{\|v\|=1}\left( \tilde{a}_1x_1+\cdots+\tilde{a}_d x_d \right).
\end{equation}

Therefore, \eqref{eq:36} can be rewritten as
\begin{equation}\label{eq:39}
\varphi(x) = \max_{y \in Y} \left( \langle x, a(y) \rangle + b(y) \right).
\end{equation}
Here the set $Y$ is defined by
\[
Y = \{\, (i,u,v) : i \in \{1,2\},\ \|u\|=1,\ \|v\|=1,\ u \in \mathbb{R}^{m+n},\ v \in \mathbb{R}^n \,\}.
\]

For any $y=(i,u,v) \in Y$, the vector $a(y)$ and scalar $b(y)$ are defined according to \eqref{eq:37}–\eqref{eq:3.8}. If $i=1$, then $a(y)$ and $b(y)$ depend only on $u$. If $i=2$, then $a(y)$ depends only on $v$ and $b(y)=0$.

Given any $x_* \in \mathbb{R}^d$, the maximizing vector $y_*=(i_*,u_*,v_*) \in Y$ in \eqref{eq:39} can be easily determined as follows. Evaluate the maximal eigenvalues of $f_1(P(x_*))$ and $f_2(P(x_*))$:  
\begin{itemize}
  \item[\textit{i})] If $\lambda_{\max}(f_1(P(x_*))) \geq \lambda_{\max}(f_2(P(x_*)))$, then $i_*=1$, $u_*$ is the corresponding unit eigenvector of $f_1(P(x_*))$, and $v_*$ is an arbitrary unit vector.  
  \item[\textit{ii})] If $\lambda_{\max}(f_1(P(x_*))) < \lambda_{\max}(f_2(P(x_*)))$, then $i_*=2$, $u_*$ is an arbitrary unit vector, and $v_*$ is the corresponding unit eigenvector of $f_2(P(x_*))$.  
\end{itemize}

The bounds $K_1$ and $K_2$ (see Introduction) for this problem can be chosen as (proofs omitted)
\[
K_1 = \sqrt{2}\,\sqrt{4\|A\|_F^2 + 2\|B\|_F^2} + \sqrt{2}, 
\qquad 
K_2 = \sqrt{\|R\|_F^2 + \|Q\|_F^2}.
\]
The compact set $X$ in (1.2) can be chosen as a box $B_a = \{\, x=(x_i): |x_i| \leq a \,\}$, where $a$ is sufficiently large.
\begin{example}\label{example5_1}
Consider the matrices
\[
    A=\begin{bmatrix}
         2 & -1 &  0 \\
        -1 &  2 & -1 \\
         0 & -1 &  2
    \end{bmatrix},\quad
    B=\begin{bmatrix}
         1 & 0 \\
         1 & 1 \\
        -2 & 1 
    \end{bmatrix},\quad
    R=\begin{bmatrix}
         1 & -1 &  2 \\
        -1 &  1 & -1 \\
         2 & -1 &  1
    \end{bmatrix},\quad
    Q=\begin{bmatrix}
         4 & -2  \\
        -2 &  3
    \end{bmatrix}.
\] 
Since $n=3$, we have $d=6$. Take $X=[-2,2]^6$. For $a=2$, the constants are $K_1=4\sqrt{10}+\sqrt{2}$, $K_2=4\sqrt{3}$, and $\sigma=75.8241$. Starting from $x^1=(1,0,0,0,0,0)\in X$, the iteration results are shown in Table~\ref{table_example6_1}.
\begin{table}[H]
\centering
\caption{Iteration results for Example~\ref{example5_1} with $a=2$}
\label{table_example6_1}
\begin{tabular}{ccccc}
\toprule
$k$ & $x^k$ & $i^k$ & $\varphi(x^k)$ & $s_k$ \\ 
\midrule
1   & $(1,0,0,0,0,0)$ & $2$ & $0$ & $-2$ \\
2   & $(0, 0, 0, 2, 0, 2)$ & $1$ & $0.58879334$ & $-2$ \\
3   & $(2, -0.574, -0.064, 2, 0, 2)$ & $1$ & $0.1465551$ & $-1.4784721$ \\
4   & $(1.086, -2, 2, 1.478, 2, 2)$ & $1$ & $12.1676481$ & $-1.4623624$ \\
5   & $(0.351, -2, -2, 1.462, 1.775, 2)$ & $1$ & $2.7302196$ & $-1.2645996$ \\
$\vdots$ & $\vdots$ & $\vdots$ & $\vdots$ & $\vdots$ \\
60  & $(2, -0.885, -1.055, 1.042, 0.452, 2)$ & $2$ & $-0.4881598$ & $-0.4867385$ \\
61  & $(2, -0.884, -1.066, 1.040, 0.460, 2)$ & $1$ & $-0.4881603$ & $-0.4881609$ \\
\bottomrule
\end{tabular}
\end{table}
With tolerance $\varepsilon=10^{-6}$, the minimum value is
\[
   \min_{x\in X}\varphi(x)\approx -0.4881603.
\]
At iteration $k=61$, $\varphi(x^{61})=-0.4881603<0$, and the associated matrix 
\[
P=P(x^{61})=
\begin{bmatrix}
 2     & -0.884 & -1.066 \\
-0.884 &  1.040 &  0.460 \\
-1.066 &  0.460 &  2
\end{bmatrix}
\]
is positive definite. For this $P$, both $\lambda_{\max}(f_1(P))<0$ and $\lambda_{\max}(f_2(P))<0$ hold, confirming that condition \eqref{eq:35} is satisfied. Thus, in this case the Riccati inequality has a solution.
\end{example}

\begin{example}\label{example5_2}
Consider the matrices
\[
    A=\begin{bmatrix}
        -4 &  3 &  1 \\
        -3 & -3 &  1 \\
         2 & -1 & -3
    \end{bmatrix},\quad
    B=\begin{bmatrix}
        1 & 0 \\
        0 & 1 \\
        1 & 1
    \end{bmatrix},\quad
    R=\begin{bmatrix}
         0 & -1 &  0 \\
        -1 & -1 & -1 \\
         0 & -1 &  0
    \end{bmatrix},\quad
    Q=\begin{bmatrix}
        2 & 0 \\
        0 & 2
    \end{bmatrix}.
\]
Since $n=3$, we have $d=6$. Take $X=[-1,1]^6$, $\varepsilon=10^{-6}$. Using the bounds defined earlier, with $a=1$ we obtain $\sigma=61.18064988$. Starting from $x^1=(1,0,0,0,0,0)\in X$, the iteration results are given in Table~\ref{table_example6_2}.
\begin{table}[!htbp]
\centering
\caption{Iteration results for Example~\ref{example5_2}}
\label{table_example6_2}
\begin{tabular}{ccccc}
\toprule
$k$ & $x^k$ & $i^k$ & $\varphi(x^k)$ & $s_k$ \\ 
\midrule
1  & $(1,0,0,0,0,0)$ & $1$ & $8.788644$ & $-19.810755$ \\
2  & $(-1,-1,1,1,-1,-1)$ & $1$ & $15.118699$ & $-11.937854$ \\
3  & $(-1,1,0.864,-1,1,-1)$ & $1$ & $7.376292$ & $-6.176726$ \\
$\vdots$ & $\vdots$ & $\vdots$ & $\vdots$ & $\vdots$ \\
34 & $(0.129,-0.077,0.082,-0.156,-0.169,0.012)$ & $2$ & $0.2618315$ & $0.2618305$ \\
35 & $(0.129,-0.077,0.082,-0.156,-0.169,0.012)$ & $1$ & $0.2618305$ & $0.2618305$ \\
\bottomrule
\end{tabular}
\end{table}
The minimum value is
\[
\min_{x\in X}\varphi(x)\approx 0.2618305.
\]
Since this strictly positive minimum is attained at an interior point of $X$, Theorem \ref{thm11} implies that enlarging $X$ (by increasing $a$) does not change the situation. In other words,
\[
  \min_{x \in \mathbb{R}^6} \varphi(x)=0.2618305>0. 
\]
Therefore, the Riccati matrix inequality has no solution in this example.
\end{example}



\begin{thebibliography}{99}

\bibitem{Aksoy2025}
B.~Aksoy, T.~B\"{u}y\"{u}kk\"{o}ro\u{g}lu, and V.~Dzhafarov,
On the numerical verification of a counterexample on parameter-dependent Lyapunov functions,
\textit{Applied Mathematics and Computation}, 492 (2025), p.~129246.

\bibitem{Bagirov2019}
A.~M.~Bagirov, G.~Ozturk, and R.~Kasimbeyli,
A sharp augmented Lagrangian-based method in constrained non-convex optimization,
\textit{Optimization Methods and Software}, 34 (2019), pp.~462--488.

\bibitem{Boyd91}
S.~Boyd and C.~Barratt,
\textit{Linear Controller Design: Limits of Performance},
Prentice Hall, Englewood Cliffs, NJ, 1991.

\bibitem{Boyd04}
S.~Boyd and L.~Vandenberghe,
\textit{Convex Optimization},
Cambridge University Press, Cambridge, 2004.

\bibitem{Drori2015}
Y.~Drori and M.~Teboulle,
An optimal variant of Kelley's cutting-plane method,
\textit{Mathematical Programming}, 160 (2016), pp.~321--351.

\bibitem{Dzhafarov22}
V.~Dzhafarov and T.~B\"{u}y\"{u}kk\"{o}ro\u{g}lu,
On one inner point algorithm for common Lyapunov functions,
\textit{Systems \& Control Letters}, 167 (2022), p.~105315.

\bibitem{Dzhafarov2015}
V.~Dzhafarov, T.~B\"{u}y\"{u}kk\"{o}ro\u{g}lu, and \c{S}.~Y{\i}lmaz,
On one application of convex optimization to stability of linear systems,
\textit{Trudy Instituta Matematiki i Mekhaniki UrO RAN}, 21 (2015), pp.~320--328.

\bibitem{HornJohnson94}
R.~A.~Horn and C.~R.~Johnson,
\textit{Matrix Analysis},
Cambridge University Press, Cambridge, 1994.

\bibitem{Jarre93}
F.~Jarre,
An interior-point method for minimizing the maximum eigenvalue of a linear combination of matrices,
\textit{SIAM Journal on Control and Optimization}, 31 (1993), pp.~1350--1377.

\bibitem{Kalinina2022}
E.~A.~Kalinina, Y.~A.~Smol'kin, and A.~Y.~Uteshev,
Stability and distance to instability for polynomial matrix families: Complex perturbations,
\textit{Linear and Multilinear Algebra}, 70 (2022), pp.~1291--1314.

\bibitem{Megiddo1984}
N.~Megiddo,
Linear programming in linear time when the dimension is fixed,
\textit{Journal of the Association for Computing Machinery}, 31 (1984), pp.~114--127.

\bibitem{Nesterov94}
Y.~Nesterov and A.~Nemirovskii,
\textit{Interior-Point Polynomial Algorithms in Convex Programming},
SIAM, Philadelphia, 1994.

\bibitem{Polyak21}
B.~T.~Polyak, M.~V.~Khlebnikov, and P.~S.~Scherbakov,
Linear matrix inequalities in control systems with uncertainty,
\textit{Automation and Remote Control}, 82 (2021), pp.~1--40.

\bibitem{Rockafellar73}
R.~T.~Rockafellar,
\textit{Convex Analysis},
Princeton University Press, Princeton, 1973.

\end{thebibliography}
\end{document}